\documentclass[12pt]{amsart}

\usepackage{amssymb}

\usepackage{anysize}
\marginsize{2cm}{2cm}{2cm}{2cm}

\usepackage[pdftex]{graphicx}

\usepackage{versions}
\usepackage{hyperref}
\hypersetup{
  colorlinks   = true, 
  urlcolor     = blue, 
  linkcolor    = blue, 
  citecolor   = red 
}
\usepackage{cancel}
\usepackage{yhmath}
\usepackage{enumitem}
\usepackage{stackengine}

\newtheorem{theorem}{Theorem}[section]
\newtheorem{proposition}[theorem]{Proposition}

\theoremstyle{definition}
\newtheorem{definition}[theorem]{Definition}

\theoremstyle{remark}
\newtheorem{remark}[theorem]{Remark}

\numberwithin{equation}{section}

\DeclareMathOperator{\rank}{rank}

\newcommand*{\where}{\ \ifnum\currentgrouptype=16 \middle\fi|\ }

\renewcommand{\epsilon}{\varepsilon}
\renewcommand{\phi}{\varphi}
\renewcommand{\kappa}{\varkappa}
\renewcommand{\theta}{\vartheta}

\def\Z{{\mathbb Z}}
\def\R{{\mathbb R}}
\def\C{{\mathbb C}}
\def\K{{\mathbb K}}

\title{Tensor rank of the determinant and periodic triangulations of $\R^n$}

\author{Sergey Avvakumov}

\author{Roman Karasev}

\address{Sergey~Avvakumov,  School of Mathematical Sciences, Tel Aviv University, Tel Aviv 69978, Israel}
\email{savvakumov@gmail.com}

\address{Roman~Karasev, Institute for Information Transmission Problems RAS, Bolshoy Karetny per. 19, Moscow, Russia 127994 and Moscow Institute of Physics and Technology, Institutskiy per. 9, Dolgoprudny, Russia 141700}
\email{r\_n\_karasev@mail.ru}
\urladdr{http://www.rkarasev.ru/en/}

\begin{document}

\maketitle

\begin{abstract}

We prove that in any $\Z^n$-periodic triangulation of $\R^n$ the number of $\Z^n$-orbits of $n$-dimensional simplices is at least the tensor rank of the $n$th determinant tensor. The latter is known to be at least $\frac{n^{n-1}}{(n-1)!}$, which is approximately $\frac{e^n}{\sqrt{2\pi n}}$ for large $n$.
The triangulation is not assumed to be geometric, meaning that its simplices can be ``curved''.

We also provide lower bounds for general spaces. A \emph{simplicial cell complex} is a CW-complex glued out of simplices with the attaching maps being simplicial embeddings; this notion generalizes simplicial complexes. We prove that if $X$ is a simplicial cell complex with cohomological classes $\alpha_i\in H^{d_i}(X;\Z_2)$ satisfying 
\[
\alpha_1 \smile \alpha_2 \smile \ldots \smile \alpha_n \neq 0,
\]
then $X$ has at least $2^n$ simplices of dimension $d_1+d_2+\ldots+d_n$. In particular, a simplicial cell complex homeomorphic to $\R P^n$, $\C P^n$, or $(S^2)^n$, has at least $2^n$ top-dimensional simplices.

A \emph{crystallization} of a manifold is a simplicial cell complex homeomorphic to this manifold and having the least possible number of vertices. We give a short explicit construction of a crystallization and a triangulation of $\R^n/\Z^n$ with $n+1$ and $2^{n+1}-1$ vertices, resp. Triangulations with this many vertices were described before and no smaller triangulation is known.

\end{abstract}

\section{Main results}

A \emph{triangulation} of a topological space is a \emph{simplicial complex} homeomorphic to the space. Informally, it is a way to describe a topological space as a set of simplices glued face-to-face.

Simplicial complexes are a special case of \emph{simplicial cell complexes} (also known as regular simplicial sets) defined below. In a simplicial cell complex each gluing map is still a simplicial embedding, but two simplices may intersect by more than one face in their boundary. A \emph{crystallization} of a space is a simplicial cell complex homeomorphic to this space and having the least possible number of vertices.

	\begin{figure}[ht]
	\center
	\includegraphics[width=0.8\linewidth]{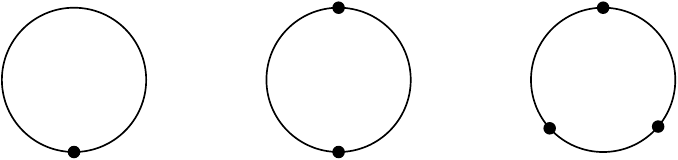}
	\caption{Three circles glued out of $1$, $2$, and $3$ simplices, resp. The middle circle is a simplicial cell complex and a crystallization. The right circle is a simplicial complex, and hence a triangulation. The left circle is not a simplicial cell complex.}
	\label{figure:circles}
	\end{figure}

Our goal is to establish new lower bounds on the minimal number of top-dimensional simplices needed to triangulate a space or to represent it as a simplicial cell complex.

For triangulations of manifolds, the current best lower bound is given by the manifold version of the $g$-theorem due to Adiprasito \cite{adiprasito2018combinatorial} combined with \cite[Theorem 6]{klee2016face}. 
It implies (see \cite[Theorem 4.4]{govc2020many}) that any triangulation of an $n$-dimensional manifold has at least roughly $\sum {n \choose i}\beta_i$ $n$-dimensional simplices, where $\beta_i$ are the Betti numbers of the manifold. This translates to exponential (in $n$) lower bounds for manifolds such as the $n$-torus $T^n=(S^1)^n$, $\R P^n$, $\C P^n$, or, even, $S^{n/2}\times S^{n/2}$. Theorem~\ref{theorem:Rn} below hints, however, that the true number should be superexponential (at least for $T^n$).

Manifold $g$-theorem is not applicable to non-manifolds or to subdivisions of manifolds which are not triangulations. So, less in known in this generality. In \cite{barany1982borsuk}, B{\'a}r{\'a}ny and Lov{\'a}sz proved that any centrally-symmetric polytope in $\R^n$ has at least $2^n$ facets. Their proof works for simplicial spheres with $\Z_2$-symmetry as well, which immediately translates to the $2^n$ lower bound on the number of $n$-simplices in any simplicial cell complex homeomorphic to $\R P^n$. In \cite{avvakumov2021topological} we generalize their result to other spaces, in particular, to the $n$-torus $T^n$. Crucially, our approach required the space to have a non-trivial fundamental group and so we could say nothing about, for example, $\C P^n$ or $(S^2)^n$.

Our first result is:

\begin{theorem}
\label{theorem:simplicial_cell_size}
Let $X$ be a topological space. Suppose $\alpha_i\in H^{d_i}(X;\Z_2)$ are its cohomological classes modulo $2$ with the property that
\[
\alpha_1 \smile \alpha_2 \smile \ldots \smile \alpha_n \neq 0.
\]
Then any simplicial cell complex homeomorphic to $X$ has at least $2^n$ simplices of dimension $d_1+d_2+\ldots+d_n$.
\end{theorem}

As a corollary, any simplicial cell complex homeomorphic to $T^n$, $\R P^n$, $\C P^n$, or $(S^2)^n$ has at least $2^n$ top-dimensional simplices. This bound is sharp for $\R P^n$. Indeed, taking the quotient of the antipodal action on the boundary of the cross-polytope in $\R^{n+1}$ one gets a crystallization of $\R P^n$ with $n+1$ vertices and $2^n$ top-dimensional simplices. By Theorem~\ref{theorem:Rn} below, this bound is not sharp for $T^n$ and is unlikely to be sharp for any other manifold.

Our second result connects the number of simplices with the tensor rank of the determinant tensor. Let us give the required definition first.

\begin{definition}[Tensor rank]
Let $V$ be a linear space over a field $\K$. For a tensor $T\in V^{\otimes n}$ its \emph{rank} $\rank_\K(T)$ is the minimal number $r$ such that $T$ can be decomposed as the sum:
\[
T=\sum^{r}_{i=1}v_{i,1}\otimes v_{i,2} \otimes \ldots \otimes v_{i,n},\quad \text { where } v_{i,j}\in V \text{ for all } j.
\]
\end{definition}

The $n$th \emph{determinant tensor} ${\det}_n$ is the determinant of $n\times n$ real matrices considered as a tensor in $(\R^n)^{\otimes n}$. The exact value of $\rank_\R({\det}_n)$ is unknown. There have been more than a few works improving on either lower or upper bounds. The current best lower \cite{han2025recursive} and upper \cite{houston2024new} bounds are 
\[
\frac{n^{n-1}}{(n-1)!} \leq \rank_\R({\det}_n) \leq B_n < \biggl( \frac{n}{e^{1-\varepsilon}\cdot {\ln n}} \biggr)^n,
\]
where $B_n$ is the $n$th Bell number. The upper bound is still very close to the trivial $n!$ bound following from the Leibniz formula. In view of the following theorem, we hope that the actual value of $\rank_\R({\det}_n)$ is closer to the above upper bound and is superexponential in $n$.

\begin{theorem}
\label{theorem:simplicial_cell_size_R}
Let $X$ be a topological space. Suppose $\alpha_1,\ldots,\alpha_n\in H^d(X;\R)$ for odd positive $d$ are its cohomological classes with the property that
\[
\alpha_1 \smile \alpha_2 \smile \ldots \smile \alpha_n \neq 0.
\]
Then any simplicial cell complex homeomorphic to $X$ has at least $\rank_\R({\det}_n)\geq \frac{n^{n-1}}{(n-1)!}$ simplices of dimension $nd$.
\end{theorem}

Sometimes this theorem can also be applied to more general subdivisions. A \emph{$\Z^n$-invariant triangulation of $\R^n$} is a simplicial complex with a free $\Z^n$-action which is $\Z^n$-equivariantly homeomorphic to $\R^n$. Equivalently, it is a tiling of $\R^n$ by $n$-simplices touching face-to-face, which is invariant under translations by the elements of $\Z^n$. Note, that the homeomorphism in the definition need not be linear on the simplices. So, the simplices of the tiling are not assumed to be \emph{geometric} and can be \emph{curved}, see Figure~\ref{figure:subdivision}.

\begin{remark}
\label{remark:distance}
For a $\Z^n$-invariant triangulation $\mathcal{T}$ of $\R^n$ let $d(\mathcal{T})$ be the minimum among edge-distances between vertices sharing the same $\Z^n$-orbit. The quotient $\mathcal{T}/\Z^n$ is always homeomorphic to $T^n$, however, depending on $d(\mathcal{T})$ we have:
\begin{itemize}
\item If $d(\mathcal{T})\geq 3$ then $\mathcal{T}/\Z^n$ is a simplicial complex, hence, a triangulation.
\item If $d(\mathcal{T})= 2$ then $\mathcal{T}/\Z^n$ is a simplicial cell complex, but not a simplicial complex.
\item If $d(\mathcal{T})=1$ then $\mathcal{T}/\Z^n$ is not a simplicial cell complex.
\end{itemize}
\end{remark}

As follows from the remark, if $d(\mathcal{T})=1$ then Theorem~\ref{theorem:simplicial_cell_size_R} cannot be directly applied to $\mathcal{T}/\Z^n$. Its conclusion still holds, however. In particular, it is true also for one vertex subdivisions of the torus, when $\mathcal{T}$ has a single vertex orbit:

\begin{theorem}
\label{theorem:Rn}
In any $\Z^n$-invariant triangulation of $\R^n$ the number of $\Z^n$-orbits of $n$-dimensional simplices is at least $\rank_\R({\det}_n)\geq \frac{n^{n-1}}{(n-1)!}$.
\end{theorem}

In Section~\ref{section:constructions} we give short and explicit constructions of a crystallization of $T^n$ with $n+1$ vertices and of a triangulation of $T^n$ with $2^{n+1}-1$ vertices. A family of triangulations of $T^n$ with $2^{n+1}-1$ vertices was first constructed in \cite{kuhnel1988combinatorial}, see also \cite{jevtic2021generalized} for a different approach. We believe (without proving it), that our construction is isomorphic to triangulations in \cite{kuhnel1988combinatorial} and \cite{jevtic2021generalized}. No triangulation of $T^n$ with fewer than $2^{n+1}-1$ vertices is known.

\subsection*{Acknowledgments} 
We thank Marcos Cossarini for suggesting the matrix \eqref{equation:2-1_matrix} when we were trying to construct a crystallization of $T^n$. We also thank ChatGPT 5 for pointing out that the lower bound in the proof of Theorem~\ref{theorem:Rn} can be stated in tensor language and is thus equal to the determinant's tensor rank.

\section{Proofs}

Recall one of the equivalent definitions of the cup product (see \cite[Section 3.2 ``Cup product'']{hatcher2002algebraic}):

\begin{definition}[Cup product]
Let $\mathcal{S}$ be a simplicial cell complex. Choose a partial order on the vertices of $\mathcal{S}$ which restricts to a complete order on every simplex. Let $a_1, a_2, \cdots, a_n$ be cocyles on $\mathcal{S}$ of dimensions $d_1,d_2,\ldots,d_n$, respectively. Then the value of their cup product on a simplex $[v_0,v_1,\ldots,v_{d_1+\ldots+d_n}]$ with $v_0<v_1<\ldots < v_{d_1+\ldots+d_n}$ is computed by the formula:
\begin{multline}
\label{multline:cup_product}
(a_1\smile a_2\smile\ldots\smile a_n)[v_0,v_1,\ldots,v_{d_1+\ldots+d_n}]:=\\
=a_1[v_0,\ldots,v_{d_1}]\cdot a_2[v_{d_1},\ldots,v_{d_1+d_2}]\cdot a_3[v_{d_1+d_2},\ldots,v_{d_1+d_2+d_3}] \cdot\ldots\cdot a_n[v_{d_1+\ldots +d_{n-1}},\ldots,v_{d_1+\ldots +d_{n}}].
\end{multline}
The cup product of the corresponding cohomology classes is:
\[
[a_1]\smile[a_2]\smile\ldots \smile [a_n] := [a_1\smile a_2\smile\ldots\smile a_n].
\]
\end{definition}

Note, that according to this definition, the cocycle $a_1\smile a_2\smile\ldots\smile a_n$ depends on the chosen partial order of the vertices. Its cohomological class $[a_1\smile a_2\smile\ldots\smile a_n]$, however, is independent of the choice. 

\begin{proof}[Proof of Theorem~\ref{theorem:simplicial_cell_size}]
Let $\mathcal{S}$ be a simplicial cell complex homeomorphic to $X$.
For each $\alpha_i$, choose a cocycle $a_i$ in $\mathcal{S}$ representing $\alpha_i$, meaning that $[a_i]=\alpha_i$.

For each $i$, construct a random $d_i$-dimensional coboundary $b_i$ in the following way: choose every $(d_i-1)$-dimensional simplex of $\mathcal{S}$ uniformly at random with probability $\frac{1}{2}$, then sum up the coboundaries of the chosen simplices. So, for any $d_i$-simplex $\sigma$ the value of $b_i(\sigma)$, and thus also the value of $(a_i+b_i)(\sigma)$, is a random variable equal to $0$ or $1$ with probabilities $\frac{1}{2}$. Also, by construction, all $b_i$ are independent of each other.

Consider an arbitrary $(d_1+d_2+\ldots+d_n)$-simplex $[v_0,v_1,\ldots,v_{d_1+\ldots+d_n}]$ with $v_0<v_1<\ldots < v_{d_1+\ldots+d_n}$. Then the values
\begin{flalign*}
&(a_1+b_1)[v_0,\ldots,v_{d_1}],\\
&(a_2+b_2)[v_{d_1},\ldots,v_{d_1+d_2}],\\
&\ldots,\\
&(a_n+b_n)[v_{d_1+\ldots +d_{n-1}},\ldots,v_{d_1+\ldots +d_{n}}]
\end{flalign*}
are independent random variables, each equal $0$ or $1$ with probabilities $\frac{1}{2}$. Hence, by (\ref{multline:cup_product}), we have that
\[(a_1+b_1)\smile (a_2+b_2)\smile\ldots\smile (a_n+b_n)[v_0,v_1,\ldots,v_{d_1+\ldots+d_n}]\] is $1$ with probability $2^{-n}$ and is $0$ otherwise.

Since 
\[[(a_1+b_1)\smile (a_2+b_2)\smile\ldots\smile (a_n+b_n)]=\alpha_1 \smile \alpha_2 \smile \ldots \smile \alpha_n \neq 0,\]
we know that for any possible choice of $b_1, b_2,\ldots, b_n$ there is at least one $(d_1+d_2+\ldots+d_n)$-simplex on which $(a_1+b_1)\smile (a_2+b_2)\smile\ldots\smile (a_n+b_n)$ evaluates to $1$. Hence, the total number of $(d_1+d_2+\ldots+d_n)$-simplices must be at least $2^n$.
\end{proof}

\begin{proof}[Proof of Theorem~\ref{theorem:simplicial_cell_size_R}]

Choose the cocycles $a_i$ representing the respective cohomology classes $\alpha_i$. 
Denote by $V$ the $n$-dimensional real vector space spanned by $a_1,\ldots,a_n$. Its natural map to $H^d(X; \R)$ is an injection. Let
\[
z = \sum_{\sigma} z_\sigma \sigma
\]
with the sum over the $nd$-simplices of $X$ be the $nd$-cycle witnessing the non-zero cup product, and we scale $z$ so that
\[
(\alpha_1\smile\alpha_2\smile\dots\smile\alpha_n)(z) = 1.
\]

For any $c_1,c_2,\ldots, c_n\in V$ we have that
\[
(c_1\smile c_2\smile\ldots\smile c_n)(z) = {\det}_n(c_1\otimes c_2\otimes\ldots\otimes c_n),
\]
where the determinant is considered in the base $a_1,\ldots, a_n$ of $V$. This is so since the cup product is skew-symmetric for odd $d$.

With a slight abuse of notation we shall view any $d$-face $e$ of $X$ as a covector in $V^*$ and denote the pairing with $c\in V$ by $e(c)$.

Order the vertices of $X$. Consider an $nd$-face $\sigma=[v_0,v_1,\ldots, v_{nd}]$ of $X$. For each $1\leq i \leq n$, let $e_i(\sigma)$ be the $d$-face $[v_{(i-1)d},\ldots, v_{id}]$. By (\ref{multline:cup_product}), we have that
\begin{multline*}
(c_1\smile c_2\smile\ldots\smile c_n)(\sigma) = e_1(\sigma)(c_1)\dots e_n(\sigma)(c_n)
= \\
= e_1(\sigma)\otimes e_2(\sigma) \otimes \ldots \otimes e_n(\sigma)(c_1\otimes c_2\otimes\ldots\otimes c_n).
\end{multline*}
Summing up this equality over all the $nd$-simplices $\sigma$ of $X$ with coefficients $z_\sigma$ we get
\begin{equation}
\label{equation:sum_over_simplices}
{\det}_n(c_1\otimes c_2\otimes\ldots\otimes c_n) = \sum_{\sigma} z_\sigma \cdot  e_1(\sigma)\otimes e_2(\sigma) \otimes \ldots \otimes e_n(\sigma)(c_1\otimes a_2\otimes\ldots\otimes c_n).
\end{equation}
This equality holds for any choice of $c_1,c_2,\ldots, c_n\in V$, so in the tensor notation
\[
{\det}_n = \sum_{\sigma} z_\sigma \cdot e_1(\sigma)\otimes e_2(\sigma) \otimes \ldots \otimes e_n(\sigma).
\]
By the definition of $\rank_\R({\det}_n)$, the number of summands on the right, and therefore the number of $nd$-simplices in $X$, is at least $\rank_\R({\det}_n)$.
\end{proof}

Let us now describe the idea of the proof of Theorem~\ref{theorem:Rn}. It does not immediately follow from Theorem~\ref{theorem:simplicial_cell_size_R}, because the quotient $\mathcal{T}/\Z^n$ is not a simplicial cell complex when $d(\mathcal{T})=1$. We construct a subdivision $\mathcal{T'}$ of $\mathcal{T}$ with $d(\mathcal{T})\geq 2$. We do it in a specific way, so that for every simplex of $\mathcal{T}$ the contribution of all but one simplices in its subdivision to the determinant \eqref{equation:sum_over_simplices} is negligible. Then we apply Theorem~\ref{theorem:simplicial_cell_size_R} to $\mathcal{T'}/\Z^n$ and notice that the number of non-negligible summands on the right hand side of \eqref{equation:sum_over_simplices} is equal tot he number of simplices in the original $\mathcal{T}/\Z^n$.

\begin{proof}[Proof of Theorem~\ref{theorem:Rn}]
For each $1\leq i \leq n$, let us define a real valued $1$-cocycle $dx_i$ on the torus $\mathcal{T}/\Z^n$ as follows. For an oriented edge $e$, the value of $dx_i(e)\in\R$ is the difference between the $i$th coordinates of the endpoint and the startpoint of one of the preimages of $e$ in $\mathcal{T}$; this value is independent on the choice of the preimage. Theorem~\ref{theorem:simplicial_cell_size_R} for $d=1$ applies to the cohomology classes $\alpha_i = [dx_i]$ if $\mathcal{T}/\Z^n$ is a simplicial cell complex. But the latter need not be the case, so we need to reduce the problem to this particular case.

Each vertex of $\mathcal{T}$ is a point in $\R^n$. Order the vertices \emph{lexicographically}, meaning that $v < u$ iff for some $i$ all the coordinates of $v$ and $u$ before the $i$th are the same, and the $i$th coordinate of $v$ is smaller than the $i$th coordinate of $u$. This lexicographic order is $\Z^n$-invariant, i.e., for any $g\in \Z^n$ we have that $v < u \iff v+g < u+g$.

Let $\mathcal{T'}$ be the barycentric subdivision of $\mathcal{T}$. In the barycentric subdivision construction the ``barycenter'' representing a face need not be the geometric barycenter of the face, but may be any point in its relative interior. Using this freedom, for every simplex (of every dimension) we choose its barycenter to be $\varepsilon$ close to its lexicographically largest vertex. By the previous remark, $\mathcal{T'}$ is also $\Z^n$-invariant, see Figure~\ref{figure:subdivision}.

	\begin{figure}[ht]
	\center
	\includegraphics[width=0.9\linewidth]{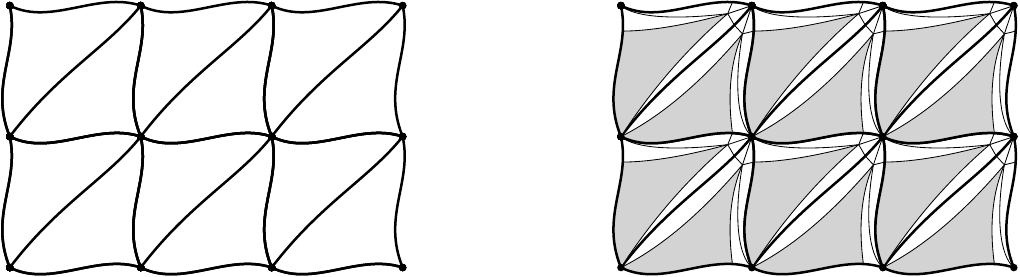}
	\caption{Subdivision $\mathcal{T'}$. Only the gray triangles contribute significantly to the volume form.}
	\label{figure:subdivision}
	\end{figure}
	
We have that $d(\mathcal{T'})\geq 2$, so the quotient $\mathcal{T'}/\Z^n$ is a simplicial cell complex and Theorem~\ref{theorem:simplicial_cell_size_R} applies.  It remains to prove, that for any choice of cocycles $c_1,c_2,\ldots,c_n$ and for every $n$-simplex $\sigma$ of $\mathcal{T}/\Z^n$, there is at most one simplex in its barycentric subdivision, whose contribution to $c_1\smile c_2\smile\ldots\smile c_n$ is not proportional to $\varepsilon$. Taking $\varepsilon\to 0$, we will see that when we sum on the right hand side of (\ref{equation:sum_over_simplices}) over the simplices of the subdivision $\mathcal{T'}/\Z^n$, each summand converges to zero, except maybe one summand per each $n$-simplex of the original $\mathcal{T}/\Z^n$.

The vertices of the barycentric subdivision $\mathcal{T'}$ are ordered by the inclusion order on the faces of $\mathcal{T}$. This is a partial \emph{dimensional} order on the vertices of $\mathcal{T'}$ and of $\mathcal{T'}/\Z^n$. We use this dimensional order to compute $c_1\smile c_2\smile\ldots\smile c_n$ on $\mathcal{T'}/\Z^n$.

Let $v_0 < v_1 < \ldots < v_n$ be the vertices of an $n$-simplex $\sigma$ of $\mathcal{T}$ ordered lexicographically. Let $\pi$ be a permutation and $\sigma_\pi$ be the simplex in the barycentric subdivision of $\sigma$ corresponding to $\pi$ as follows. Ordered dimensionally, the vertices of $\sigma_\pi$ are the barycenters of the following faces of $\sigma$:
\[
\{v_{\pi(0)}\}, \{v_{\pi(0)}, v_{\pi(1)}\}, \{v_{\pi(0)}, v_{\pi(1)}, v_{\pi(2)}\}, \ldots, \{v_{\pi(0)}, v_{\pi(1)}, v_{\pi(2)},\ldots, v_{\pi(n)}\} .
\]
So, when computing $c_1\smile c_2\smile\ldots\smile c_n$ on $\sigma_\pi$ using the formula (\ref{multline:cup_product}), we will multiply the values of the cocycles on the following edges:
\begin{flalign*}
&\{v_{\pi(0)}, v_{\pi(1)}\}, \{v_{\pi(0)}\} \\
&\{v_{\pi(0)}, v_{\pi(1)}, v_{\pi(2)}\}, \{v_{\pi(0)}, v_{\pi(1)}\} \\
&\ldots\\
&\{v_{\pi(0)}, v_{\pi(1)}, v_{\pi(2)},\ldots, v_{\pi(n)}\}, \{v_{\pi(0)}, v_{\pi(1)}, v_{\pi(2)},\ldots, v_{\pi(n-1)}\}.
\end{flalign*}
For any $1\leq i \leq n$ consider the edge
\[
\{v_{\pi(0)}, v_{\pi(1)}, v_{\pi(2)},\ldots, v_{\pi(i)}\}, \{v_{\pi(0)}, v_{\pi(1)}, v_{\pi(2)},\ldots, v_{\pi(i-1)}\}.
\]
By our choice of the barycenters positions, the endpoints of this edge are $\varepsilon$ close to the lexicographically largest vertex among $\{v_{\pi(0)}, v_{\pi(1)}, v_{\pi(2)},\ldots, v_{\pi(i)}\}$ and $\{v_{\pi(0)}, v_{\pi(1)}, v_{\pi(2)},\ldots, v_{\pi(i-1)}\}$, respectively. Which means that, unless $v_{\pi(i)}$ is lexicographically larger than any of the vertices $v_{\pi(0)}, v_{\pi(1)}, v_{\pi(2)},\ldots, v_{\pi(i-1)}$, the edge is $2\varepsilon$ short and so the cup product on $\sigma_\pi$ goes to $0$ when $\varepsilon$ goes to $0$.

Using this consideration for every $1\leq i \leq n$, we see that the cup product on $\sigma_\pi$ is $O(\varepsilon)$ unless $v_{\pi(i)}$ is lexicographically larger than any of $v_{\pi(0)}, v_{\pi(1)}, v_{\pi(2)},\ldots, v_{\pi(i-1)}$ for every $i$. Which is only possible if $\pi={\rm id}$. So, among all the simplices $\sigma_\pi$ in the barycentric subdivision of $\sigma$, only $\sigma_{\rm id}$ contributes to (\ref{equation:sum_over_simplices}) something that does not tend to zero as $\varepsilon\to 0$.

In the (pointwise in terms of $c_1,c_2,\ldots, c_n$) limit $\varepsilon\to 0$ the right hand side of (\ref{equation:sum_over_simplices}) for $\mathcal T'/Z_n$ is a sum of decomposable tensors one-to-one corresponding to the $n$-simplices of $\mathcal T/\Z_n$. Thus the result follows.
\end{proof}

\section{Crystallization and triangulation of $T^n$}
\label{section:constructions}

Denote by $\mathcal{T}$ the most standard $\Z^n$-invariant triangulation of $\R^n$.
In it, every unit cube with integer vertices is subdivided into $n!$ congruent simplices using the \emph{staircase} triangulation. For example, the standard unit cube $[0,1]^n$ is subdivided into the $n!$ simplices:
\[
[0, e_{\pi(1)}, e_{\pi(1)}+e_{\pi(2)}, \ldots, e_{\pi(1)}+e_{\pi(2)}+\ldots+ e_{\pi(n)}],
\]
where $e_1,e_2,\ldots,e_n$ is the standard basis and $\pi$ ranges through all the permutations.

The only property of $\mathcal{T}$ we will need is that its every edge is a $\{0,1\}$ vector, i.e., each of its coordinates belongs to $\{0,1\}$.

Let $A\subset\Z^n$ be the subgroup of integer vectors with the sum of their coordinates divisible by $n+1$.

\begin{proposition}
\label{proposition:crystallization}
The quotient $\mathcal{T}/A$ is a crystallization of $T^n$.
\end{proposition}
\begin{proof}
The subgroup $A$ is a lattice isomorphic to $\Z^n$ and the quotient is homeomorphic to the torus $T^n$.

All the edges of $\mathcal{T}$ are $\{0,1\}$ vectors, the sum of their coordinates is between $0$ and $n$, which is only divisible by $n+1$ if the vector is $0$. So, no two adjacent vertices of $\mathcal{T}$ lie in the same $A$-orbit which makes $\mathcal{T}/A$ a simplicial cell complex.

Finally, $\mathcal{T}/A$ is a crystallization as it only has $n+1$ vertices.
\end{proof}

Note that the subgroup $A$ may be generated by the rows of the matrix
\begin{equation}
\label{equation:2-1_matrix}
\begin{bmatrix}
2 & 1 & 1 & \cdots & 1 \\
1 & 2 & 1 & \cdots & 1 \\
1 & 1 & 2 & \cdots & 1 \\
\vdots & \vdots & \vdots & \ddots & \vdots \\
1 & 1 & 1 & \cdots & 2
\end{bmatrix}.
\end{equation}

To construct a triangulation of $T^n$ with $2^{n+1}-1$ vertices we again take a quotient of $\mathcal{T}$, this time by the subgroup generated by the rows of the matrix:
\[
B:=\begin{bmatrix}
1 & 0 & 0 & \cdots & 0 & -2 \\
0 & 1 & 0 & \cdots & 0 & -4 \\
0 & 0 & 1 & \cdots & 0 & -8 \\
\vdots & \vdots & \vdots & \ddots & \vdots \\
0 & 0 & 0 & \cdots & 1 & -2^{n-1} \\
0 & 0 & 0 & \cdots & 0 & 2^{n+1}-1
\end{bmatrix}.
\]
The upper-left $(n-1)\times(n-1)$ minor of $B$ is the identity matrix, the rightmost column is $-2, -4, -8, -\ldots, -2^{n-1}, 2^{n+1}-1$, the remaining entries are zeroes.

\begin{proposition}
The quotient $\mathcal{T}/B$ is a triangulation of $T^n$ with $2^{n+1}-1$ vertices.
\end{proposition}
\begin{proof}
The number of vertices is the determinant of $B$.

It remains to check that $\mathcal{T}/B$ is a triangulation. This will follow from the fact that the edge-distance between any two vertices in the same $B$-orbit is at least $3$. Equivalently, we need to prove that an integer linear combination of the rows of $B$ cannot be represented as the sum or the difference of two $\{0,1\}$ vectors; the latter is either a $\{0,1,2\}$ vector or a $\{-1,0,1\}$ vector.

Let $v$ be the sum of the rows $r_1, r_2, \ldots, r_n$ of $B$ taken with integer coefficients $c_1, \ldots, c_n$.
The first $n-1$ coordinates of $v$ are exactly $c_1, c_2, \ldots, c_{n-1}$.

It follows that $\{c_1, c_2, \ldots, c_{n-1}\}\subseteq \{-1,0,1,2\}$. So, the absolute value of the last coordinate of $c_1r_1+\ldots+c_{n-1}r_{n-1}$ is at most $2\cdot(2+4+\ldots+2^{n-1})=4\cdot(2^{n-1}-1)=2^{n+1}-4$. On the other hand, the last coordinate of $r_n$ is $2^{n+1}-1$, which is larger by $3\notin \{-1,0,1,2\}$. So, $c_n$ has to be $0$, otherwise the last coordinate of $v$ will be out of bounds.

Suppose that one of the coordinates of $v$ is $2$. Then $\{c_1, c_2, \ldots, c_{n-1}\}\subseteq \{0,1,2\}$. Which means that, since $c_n=0$, the last coordinate of $v$ is strictly negative, which is impossible since $v$ cannot have both $2$ and $-1$ among its coordinates.

We may now assume that $v$ has no $2$s among its coordinates. Then $\{c_1, c_2, \ldots, c_{n-1}\}\subseteq \{-1,0,1\}$ and the last coordinate of $v$ is the sum of different positive powers of $2$ with coefficients in $\{-1,0,1\}$. Which means that its absolute value is at least $2$, contradiction.
\end{proof}

Note, that $2^{n+1}-1$ is exactly the degree of a vertex in $\mathcal{T}$ plus $1$. So, no smaller triangulation of $T^n$ can be constructed by taking a quotient of $\mathcal{T}$:

\begin{proposition}
Suppose $C$ is a non-degenerate integer $n\times n$ matrix with the property, that no non-zero $\{0,1,2\}$ or $\{-1,0,1\}$ vector can be represented as an integer combination of its rows. Then $\left|{\det}(C)\right|\geq 2^{n+1}-1$.
\end{proposition}

So, to get a smaller triangulation of $T^n$ one needs to start with a different $\Z^n$-invariant triangulation of $\R^n$ with smaller vertex degrees and not too many vertex orbits.

\bibliography{source}
\bibliographystyle{abbrv}

\end{document}